\documentclass{article}
\usepackage[latin1]{inputenc}
\usepackage{amsmath}
\usepackage{amsfonts}
\usepackage{amssymb}
\usepackage{amsthm}
\usepackage{mathrsfs}
\usepackage{lscape}
\usepackage{euscript}
\usepackage{latexsym,enumerate,color}
\usepackage{tikz}

\allowdisplaybreaks

\begin{document}

\def\red{\textcolor{red}}

\theoremstyle{plain} \newtheorem{Thm}{Theorem}[section]
\newtheorem{prop}[Thm]{Proposition}
\newtheorem{lem}[Thm]{Lemma}
\newtheorem{step}{Step}
\newtheorem{de}{Definition}
\newtheorem{obs}{Observation}
\newtheorem{cor}[Thm]{Corollary}

\newcommand{\kor}{\mathring }
\newcommand{\Korz}{\mathring{\mathbb{Z}} }
\newcommand{\korx}{\mathring{x} }
\newcommand{\res}{\restriction }
\newcommand{\kalapz }{\hat{\mathbb{Z}} }
\theoremstyle{remark} \newtheorem*{pf}{Proof}
\renewcommand\theenumi{(\alph{enumi})}
\renewcommand\labelenumi{\theenumi}
\renewcommand{\qedsymbol}{}
\renewcommand{\qedsymbol}{\ensuremath{\blacksquare}}
\newcommand{\mcomment}[1]{\textcolor{red}{#1}}

\newcommand{\bigslant}[2]{{\raisebox{.2em}{$#1$}\left/\raisebox{-.2em}{$#2$}\right.}}

\newcommand{\rad}{{\sf Rad}}
\def\und{\underline}
\def\N{\mathbb{N}}
\def\Q{\mathbb{Q}}
\def\Z{\mathbb{Z}}
\def\C{\mathbb{C}}
\newcommand{\Ld}{|\!\!|}
\newcommand{\Rd}{|\!\!|}
\newcommand{\dsl}{/\!\!/}
\newcommand{\cala}{\mathcal }
\newcommand{\etha}{\alpha}
\newcommand{\embedding}{\hookrightarrow}
\newcommand{\semidirect}{\rtimes}
\newcommand{\sr}{S-ring }
\newcommand{\pp}{\mbox {\bf P}\,}
\newcommand{\pr}{\mbox {\bf Proof.\ }}
\newcommand{\aaa}{\mbox {\cala A}}
\newcommand{\pset}{\mbox {\{p_1,...,p_m\}}}
\newcommand{\zn}{\mbox {\bf Z}_n}
\newcommand{\pg}{permutation group }
\newcommand{\pgs}{permutation groups }
\newcommand{\srg}[3]{${\cala #1}=<{\underline #2}_0,...,{\underline
#2}_#3>$}
\newcommand{\isom}{\cong}
\newcommand {\sth}{ such that }
\newcommand {\imp}{ imprimitivity system }
\newcommand {\tram}{ $V(C_n,G_1)$ }
\newcommand {\pair}[1] {#1_1 #1_2^{-1}}
\newcommand {\cnp}{ C_{n/p} }
\newcommand {\normal}{\trianglelefteq}
\newcommand {\congruence} {\equiv}
\newcommand {\wreath}{\wr}
\newcommand {\longline}{\left|\right.}
\newcommand {\Sup}[1]{{\sf Sup}({#1})}
\newcommand {\Aut}[1]{{\sf Aut}({#1})}
\newcommand {\Out}[1]{{\sf Out}({#1})}
\newcommand {\Inn}[1]{{\sf Inn}({#1})}
\newcommand{\aalg}{{\sf Aut}_{alg}}
\newcommand {\New}[1]{{\sl #1}}
\newcommand {\Sym}[1]{{\sf Sym }(#1)}
\newcommand {\encodes}{\ encodes\ }
\newcommand {\tran}[1]{{#1}^*}
\newcommand {\twomin}{({\sf Sup}_2(H_R))^{min}}
\newcommand{\Log}[2]{{\mbox{\boldmath$\eta$\unboldmath}}_{#1}(#2)}
\newcommand {\Bsets}[1]{{\sf Bsets}(#1)}
\newcommand {\tranjugate}{tranjugate\ }
\newcommand {\cclosed}{ complete \ }
\newcommand{\CCI}{ {${\rm CI}^{(2)}$} }
\newcommand{\fp}{{\Zset}_p}
\newcommand{\HH}{{\sf H}}
\newcommand{\GG}{{\sf G}}
\newcommand{\m}[1]{\overline{#1}}
\newcommand{\Ind}{[\m{0},\m{p-1}]}
\newcommand{\cay}[2]{{\sf Cay}(#1,#2)}
\newcommand{\laa}{\langle\!\langle}
\newcommand{\raa}{\rangle\!\rangle}
\newcommand{\lla}{\langle\!\langle}
\newcommand{\rra}{\rangle\!\rangle}
\newcommand{\A}{{\cala A}}
\newcommand{\aut}[1]{{\sf Aut}(#1)}

\newcommand{\bsets}[1]{\Bsets{#1}}
\newcommand{\myplus}{\biguplus} 
\newcommand{\Stab}{{\sf Stab}}
\newcommand{\spn}{{\sf Span}}
\newcommand{\Ss}{{\cala S}}
\newcommand{\mymid}{\,\,\vline\,\,}
\newcommand{\sfM}{{\sf M}}
\newcommand{\cM}{{\mathcal M}}
\newcommand{\F}{{\mathbb F}}
\newcommand{\cR}{{\mathcal R}}
\newcommand{\cP}{{\mathcal P}}
\newcommand{\cF}{{\mathcal F}}
\newcommand{\cL}{{\mathcal L}}
\newcommand{\cT}{{\mathcal T}}
\newcommand{\cB}{{\mathcal B}}
\newcommand{\cC}{{\mathcal C}}
\newcommand{\cD}{{\mathcal D}}
\newcommand{\e}{{\sf e}}
\newcommand{\fix}{{\sf Fix}}
\newcommand{\trace}{{\rm trace}}

\newcommand{\sym}[1]{{\sf Sym }(#1)}
\newcommand{\cel}{{\rm Cel}}
\newcommand{\lm}{\lambda}
\newcommand{\SP}{{\rm Sp}}
\newcommand{\ot}[1]{{\bf O}^\theta(#1)}
\newcommand{\wt}[1]{\widetilde{#1}}
\newcommand{\wh}[1]{\widehat{#1}}
\newcommand{\btwn}[2]{{}_{#1}\!S_{#2}}
\newcommand{\haho}{half-homogeneous}
\newcommand{\coco}{coherent configuration}
\newcommand{\cocos}{coherent configurations}
\newcommand{\irr}{{\rm Irr}}
\newcommand{\rk}{{\rm rk}}
\newcommand{\cS}{{\mathcal S}}
\newcommand{\ea}{\Z_2^3}
\newcommand{\sg}[1]{\langle{#1}\rangle}
\newcommand{\ovr}[1]{\overline{#1}}
\newcommand{\tr}{T}
\newcommand{\vpphi }{\varphi}
\newcommand{\ord}{{\rm ord}}
\newcommand{\cyc}{\ cyclotomic}
\newcommand{\cA}{{\mathcal A}}
\newcommand{\Ker}{{\rm Ker}}
\newcommand{\End}{{\rm End}}
\newcommand{\Fr}{{\rm Fr}}
\newcommand{\fq}{{\mathbb F}_q}
\newcommand{\sig}{\sigma }
\newcommand{\cN}{{\mathcal N}}
\newcommand{\veps}{\varepsilon}
\newcommand{\cd}{{\rm c.d.}}
\newcommand{\mtrx}[4]{\left(\begin{array}{cc} #1 & #2 \\ #3 & #4\end{array}\right)}
\newcommand{\irreps}{ irreducible representations }

\newcommand{\core}{{\sf core}}
\def\f{\EuScript}
\newcommand{\comment}[1]{}
\newcommand{\au}{{\sf Aut}}

\newcommand\restr[2]{\ensuremath{\left.#1\right|_{#2}}}

\def\Zq{\mathbb{Z}_q}
\def\Zp{\mathbb{Z}_p^3}
\def\Zpq{\mathbb{Z}_q \times \mathbb{Z}_p^3 }
\def\hatZpq{\hat{\mathbb{Z}}_q \times \hat{\mathbb{Z}}_p^3}
\def\ringZpq{\mathring{\mathbb{Z}}_q \times \mathring{\mathbb{Z}}_p^3}
\def\stb{,\ldots ,}
\def\bfs{\bfseries}

\title{The Cayley isomorphism property for Cayley maps}
\author{Mikhail Muzychuk, G\'abor Somlai }
\date{\today}
\maketitle

\begin{abstract} In this paper we study finite groups which have Cayley isomorphism property with respect to Cayley maps, CIM-groups for a brief. We show that the structure of the CIM-groups is very restricted. It is described in Theorem~\ref{111015a} where a short list of possible candidates for CIM-groups is given. Theorem~\ref{111015c} provides concrete examples of infinite series of CIM-groups.
\end{abstract}
\section{Introduction}

Let $H$ be a finite group and $S$ a subset of $H\setminus \{ 1\}$. A {\it Cayley (di)graph}
$\cay{H}{S}$ is defined by having the vertex set $H$ and $g$ is adjacent
to $h$ if and only if $g^{-1} h \in S$. The set $S$ is called the
{\it connection set} of the Cayley graph $\cay{H}{S}$. A Cayley graph
$\cay{H}{S}$ is {\it undirected} if and only if $S=S^{-1}$, where $S^{-1} =
\left\{ \, s^{-1} \in H \mid s \in S \, \right\} $.
Every left multiplication via elements of $H$ is an automorphism of
$\cay{H}{S}$, so the automorphism group of every Cayley graph over $H$
contains a regular subgroup isomorphic to $H$. Moreover, this property
characterises the Cayley graphs of $H$.
The group consisting of the elements of the left multiplications will be denoted by $\wh{H}$ and the left multiplication with $h \in H$ by $\wh{h}$ (that is $\wh{h}(x)=hx$).
Finally, a {\it Cayley map} $\cay{H,S}{\rho}$ is an undirected Cayley graph $\cay{H}{S}$ endowed with a cyclic ordering $\rho\in\sym{S}$ of the connection set.

We say that a map $\cay{H,S}{\rho}$ is  {\it connected} if the underlying Cayley graph is connected, that is $\sg{S}=H$.

Using a less combinatorial approach, a Cayley map is a $2$-cell embedding of a Cayley graph into oriented surface with the same cyclic rotation around each vertex. For precise definiton of embedding graphs into orientable surfaces, see \cite{jonessingerman}.
Several different subclasses of Cayley maps have been investigated. The notion of a Cayley map first appeared in the paper of Biggs \cite{biggs} who investigated balanced Cayley maps. A Cayley map $\cay{H,S}{\rho}$ is called {\it balanced} if $\rho(s^{-1})= \rho(s)^{-1}$ and it is called {\it antibalanced} if
$\rho(s^{-1})= \rho^{-1}(s)^{-1}$. Further, a Cayley map $M$ is called {\it regular} if its automorphism group is transitive on the arcs as well. Following Jajcay and Siran \cite{jajcaysiran}, we say that for a group $H$ a permutation $\phi\in\sym{H}$  is a {\it skew-morphism} if there exists a mapping $\pi : H \mapsto \N$ such that $\phi(gh)=\phi(g) \phi^{\pi(g)}(h)$ for every $g,h \in G$.

Given two Cayley maps $M_1=\cay{H_1,S_1}{\rho_1}$ and $M_2=\cay{H_2,S_2}{\rho_2}$, a bijection $\phi:H_1\rightarrow  H_2$ is a {\it map isomorphism} from
$M_1$ to $M_2$ if $\phi$ is an isomorphism of the underlying Cayley graphs and for all $h \in H_1,s\in S_1$ it holds that $\phi(h)^{-1}\phi(h\rho_1(s)) = \rho_2(\phi(h)^{-1}\phi(hs))$. 
Denoting by $\Delta_h\phi$ the "differential" map $s\mapsto \phi(h)^{-1}\phi(hs), s\in S_1$ one can rewrite the latter condition as follows $(\Delta_h\phi) \rho_1 = \rho_2(\Delta_{\phi(h)}\phi)$. Notice that since $\phi$ is a graph isomorphism the map $\Delta_h\phi$ is a bijection between $S_1$ and $S_2$ for every $h\in H$. 

In what follows we say that $M_1$ and $M_2$ are {\it Cayley isomorphic} if there exists a group isomorphism $\phi:H_1\rightarrow H_2$ which is simulteneously a map isomorphism, that is $\phi(S_1)=S_2$ and
$\phi(\rho_1(s))=\rho_2(\phi(s))$ holds for each $s\in S_1$.

 The {\it automorphism group} of a Cayley map $M=\cay{H,S}{\rho}$ is the set of
all isomorphisms from $M$ to $M$ and it will be denoted by $\au(M)$.
It is clear that $\widehat{H}\leq\au(M)$. Thus $\au(\cay{H,S}{\rho})$ contains the regular subgroup $\wh{H}$.
Every group automorphism $\sigma\in\aut{H}$  induces Cayley isomorphism between the maps $\cay{H,S}{\rho}$ and $\cay{H,\sigma(S)}{\sigma' \rho (\sigma')^{-1}}$ where $\sigma'=\sigma|_S$ is the restriction of $\sigma$ on $S$. Thus a group automorphism $\sigma$ is an automorphism of a map $\cay{H,S}{\rho}$ if and only if $\sigma(S)=S$ and $\sigma|_S \rho = \rho \sigma|_S$. Since $\rho$ is a full cycle, the latter condition is equivalent to $\sigma|_S = \rho^k$ for some integer $k$.

The so-called {\it CI (Cayley isomorphism)} property of groups is well studied with respect to  Cayley graphs. A group $H$ is called a {\it CI-group} with respect to graphs (CIG-groups, for short) if two Cayley graphs of $H$ are isomorphic if and only if they are isomorphic by a group automorphism as well. A Cayley graph $\Gamma = \cay{H}{S}$ is called a {\it CI-graph} if every Cayley graph $\cay{H}{T}$ isomorphic $\Gamma$ is Cayley isomorphic to $\Gamma$.
For an old but excellent survey about CI-groups, see \cite{li} and further results can be found in \cite{liluppp}.
Similarly to the original definition of the CI property we say that a Cayley map $M= \cay{H,S}{\rho}$ is a {\it CI-map} of $H$ if every Cayley map $M'$ over $H$ isomorphic to $M$ is also Cayley isomorphic to $M$. We call a group $H$ a {\it CIM-group} if for every Cayley map $\cay{H,S}{\rho}$ is a CI-map.

A Cayley map $M=\cay{H,S}{\rho}$ can also be considered as a ternary relational structure on the vertices of the underlying graph. Three vertices $(x,y,z)$ are in the relation $\mathcal{R}$
if and only if $x^{-1}y, x^{-1}z\in S$ and $\rho(x^{-1}y) =x^{-1}z$. The automorphism group $\aut{M}$ consists of all those permutations of the vertices which preserves the relation $\mathcal{R}$. In particular, it is a $3$-closed permutation group.
This observation allows us to use the technique developed by Babai to solve problems concerning CIM-groups. Moreover a theorem of P\'alfy \cite{ppp} shows that the groups which are CI-groups for every $m$-ary relational structures are the cyclic groups of order $n$, where $(n,\phi(n))=1$ and the Klein group. P\'alfy also proved that if a group is not a CI-group with respect to some $m$-ary relation, then it is not a CI-group with respect to $4$-ary relational structures.

CI-groups with respect to ternary relations (CI$^{(3)}$-groups, for short) were investigated by Dobson \cite{dobson},\cite{dobson2} and later by Dobson and Spiga \cite{dobsonspiga}. Although the class of CI$^{(3)}$-groups is rather narrow, its full classification is not finished yet. The latest results may be found in \cite{dobson2} and \cite{dobsonspiga}. Since map automorphism group is $3$-closed, each CI$^{(3)}$-group is a CIM-group. The converse is not true. For example, every elementary abelian $2$-groups of rank at least $6$ is a CIM-group but not a CI$^{(3)}$-group.

As it was also pointed out by Dobson and Spiga \cite{dobsonspiga} every CI$^{(3)}$-group is also a CI$^{(2)}$-group, that  is a group which has a  the CI-property with respect to binary relational structures.
However, we will prove that there are CIM-groups which are not CI$^{(2)}$-groups. The Venn diagram below reflects the relationships between the three classes of CI-groups.

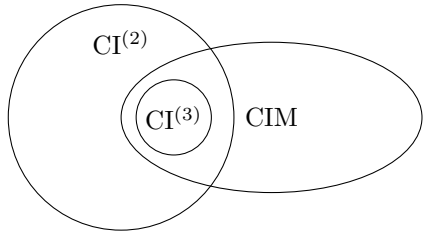
\begin{center}
\begin{tikzpicture}
\draw {(0.7,0) circle (0.5cm) node {CI$^{(3)}$}};
\draw {(0,0) circle (1.5cm)};
\draw {(0.0,1.0) node {CI$^{(2)}$}};
\draw {(2.0,0) ellipse (2cm and 1cm)  node {CIM}};
\end{tikzpicture}
\end{center}

Our first result formulates necessary conditions for being a CIM-group.
\begin{Thm}\label{111015a} Let $H$ be a CIM-group. Then $H$ is isomorphic to one of the following groups
\begin{enumerate}
\item $\Z_m\times \Z_2^r,\ \Z_m\times \Z_4,\ \Z_m\times \Z_8,\  \Z_m\times Q_8 $;
\item  $\Z_m\rtimes \Z_{2^e}, e=1,2,3$,
\end{enumerate}
where $m$ is an odd square-free number.
\end{Thm}

The second main result provides several infinite series of  CIM-groups.
\begin{Thm}\label{111015c} The following groups are CI-groups with respect to Cayley maps.
$$
\Z_m\times \Z_2^r,\ \Z_m\times \Z_4,\  \Z_m\times Q_8
$$
where $m$ is an odd square-free number.
\end{Thm}
As an immediate corollary of the above Theorems we obtain the following criterion.
\begin{Thm}\label{thm1}
A group $H$ of odd order is a  CIM-group if and only if $H$ is a cyclic group of a square free order.
\end{Thm}

Notice that obtained results do not provide a complete classification of cyclic CIM-group. This is because we do not know which of the groups $\Z_m\times\Z_8$, $m$ is odd and square-free, are CIM-groups.
Proposition~\ref{111015b} shows that $\Z_8$ is a CIM-group. We believe that all groups of the above structure have the CIM-property.

Our paper is organised as follows. In Section \ref{sectiongeneral} we collect a few general results about CI-property which will be used later. In Section \ref{sectionnegative} we characterize Sylow subgroups of CIM-groups. Section \ref{sectionpositive} is devoted to the proof of Theorem~\ref{111015a}. The last section provides proofs of  Theorems~\ref{111015c} and \ref{thm1}.

Most of the group-theoretical notation used in the paper are standard  and can be found in~\cite{W64}.

\section{General observations}\label{sectiongeneral}
The original CI property for graphs is inherited by subgroups which gives us a strong  tool to determine the list of possible CI-groups. Similar, but a weaker, property holds for CIM-groups as well. Let us call a group $H$ to be a {\it connected} CIM-group if it is a CI-group with respect to connected maps.
\begin{lem}\label{lemsubgroup}
Every subgroup of a CIM-group is a connected CIM-group.
\end{lem}
\begin{proof}
Let $G$ be a CIM-group and $H \le G$. Let us assume that $\cay{H,S}{\rho}$ and $\cay{H,S'}{\rho'}$ are isomorphic connected Cayley maps of $H$.
 Let $f$ be a map isomorphism from $\cay{H,S}{\rho}$ to $\cay{H,S'}{\rho'}$. Then $\wh{g_2} f \wh{g_1}^{-1}$ is an isomorphism between the connected component  of $\cay{G,S}{\rho}$ on $g_1H$ and the one $\cay{G,S'}{\rho'}$ on $g_2H$. This shows that the connected components of $\cay{G,S}{\rho}$ and $\cay{G,S'}{\rho'}$ are isomorphic. Therefore $\cay{G,S}{\rho}$ and $\cay{G',S'}{\rho'}$ are isomorphic Cayley maps.
Since $G$ is a CIM-group there exists $\alpha \in \au(G)$, which is an isomorphism from $\cay{G,S}{\rho}$ to $\cay{G,S'}{\rho'}$. Since the  Cayley map $\cay{H,S}{\rho}$ is a connected component of $\cay{G,S}{\rho}$, its image $\cay{\alpha(H),\alpha(S)}{\alpha|_S\rho(\alpha|_S)^{-1}}$ is a connected component of $\cay{G,S'}{\rho'}$. Therefore $\alpha(H)$ is a left coset of $H$ implying $\alpha(H)=H$. Hence $\alpha|_H$ is a Cayley isomorphism between the above maps.
\end{proof}
This result suggests that it is worth investigating $p$-groups which arise as the Sylow $p$-subgroups of finite groups.

Another important observation is that if $\cay{H,S}{\rho}$ is a Cayley map with $|S| \le 2$, then the Cayley graph $\cay{H,S}$ has to be a CI-graph since there exists only one cyclic ordering on one or two elements. This shows that the automorphism group of a CIM-group $H$  has only one orbit on the elements of order $2$ and for every $g,h \in H$ with the same order there exists $\alpha \in \au(H)$ with $\alpha(g) =h$ or $\alpha(g) =h^{-1}$. Groups having this property were investigated by Li and Praeger \cite{lipraeger}.

The following lemma is due to Babai \cite{babai} and applies for every Cayley relational structures.
\begin{lem}[Babai]\label{babai}
Let $\cay{H}{\mathcal{R}}$ be a Cayley relational structure. Then $\cay{H}{\mathcal{R}}$ has the CI-property if and only if for every regular subgroup $\mathring{H}$ of $\au(\cay{G}{\mathcal{R}})$ there exists $\alpha \in \au(\cay{G}{\mathcal{R}})$ with $\alpha(\mathring{H}) = \wh{H}$.
\end{lem}
In what follows we refer to a regular permutation subgroup isomorphic to $H$ as $H$-{\it regular}  subgroup.

The statement below describes the structure of the Cayley map automorphism group. Although it was proven by Jajcay \cite{jajcay} we prefer to provide its proof here to make the paper self-contained.
\begin{lem}\label{310116a} Let $M:=\cay{H,S}{\rho}$ be a connected 
Cayley map and $G:=\aut{M}$ its automorphism group.
Then $G_e$ acts faithfully on $S$ and its restriction $(G_e)|_S$ is contained in $\sg{\rho}$. In particular, $G_e$ is cyclic.
\end{lem} 
\begin{proof} Pick an arbitrary $\phi\in G_e$. Then $\Delta_e \phi = \phi|_S$ implying $\rho \phi|_S =\phi|_S \rho$. Since $\rho$ is a full cycle on $S$, any permutation commuting with it belongs to $\sg{\rho}$. Therefore 
$(G_e)|_S\leq\sg{\rho}$. This inclusion also implies that for each $s\in S$ the two-point stabilizer $G_{e,s}$ acts trivially on $S$. Therefore $G_{h,hs}$ acts trivially on $hS$ for any $h\in H$ and $s\in S$. Thus if $\phi$ fixes $e$ and $s\in S$, then it fixes pointwise the sets $S, S^2, S^3$ etc. Since $\cay{H}{S}$ is connected, we conclude that $G_{e,s}$ is trivial, i.e. $G_e$ acts faithfully on $S$. 
\end{proof}

The above statement shows the full automorphism group $G$ of a connected $\cay{H,S}{\rho}$ is a product of $\wh{H}$ with the cyclic group $G_e$. Moreover the restriction of $G_e$ on $S$ is contained in $\sg{\rho}$.

\section{Sylow subgroups of CIM-groups}\label{sectionnegative}
Similarly to the classical case of CI-groups, it follows from Lemma \ref{lemsubgroup} that it is important to investigate $p$-groups.
Babai and Frankl proved that if a group $H$ is a $CI^{(2)}$-group of prime power order, then $H$ is either elementary abelian $p$-group, the quaternion group of order $8$ or a cyclic group of small order.   The statement below describes odd order Sylow subgroups of a $CIM$-group.

\begin{lem}\label{sylowodd}
A Sylow $p$-subgroup of a CIM-group $H$ corresponding to an odd prime divisor $p$ of $|H|$ has order $p$.
\end{lem}
\begin{proof} It follows from Lemma~\ref{lemsubgroup} that it is sufficient to show that any subgroup of order $p^2$ is not a connected CIM-subgroup.

Let $K$ be a group of order $p^2$. Then either $K\cong\Z_p^2$ or $K\cong\Z_{p^2}$. In both cases there exists an automorphism $\beta\in\aut{K}$ of order $p$ (the concrete examples of $\beta$ are given below). A direct check shows that  the bijection $\alpha\in\sym{K}$ defined via $\alpha(x)=-\beta(x)$ is an automorphism of $K$ of order $2p$. It follows from $\alpha^p = -1$ that each non-zero $\alpha$-orbits is symmetric, and, therefore, has even cardinality. This implies that at least one orbit of $\alpha$ contains $2p$ element. Let us denote this orbit as $S$. Clearly $\sg{S}=K$. Consider a Cayley map $M=\cay{K,S}{\alpha|_K}$. The group $G:=\aut{M}$ contains the semidirect product $\wh{K}\rtimes\sg{\alpha}\leq\sym{K}$. Combining this with $|\au(M)|\leq |K||S|=|K||\sg{\alpha}|$ we conclude that  $G= \wh{K}\rtimes\sg{\alpha}$ (so, $M$ is a balanced regular map). We claim that $M$ is not a CI-map. According to Lemma~\ref{babai} it is enough to find two $K$-regular subgroups of $G$ which are not conjugate in $G$. Since $\wh{K}$ is normal in $G$, it is sufficient to find a $K$-regular subgroup of $G$ distinct from $\wh{K}$.   To point out such a subgroup we consider the cases of $K\cong\Z_{p^2}$ and $K\cong \Z_p^2$ separately. In both cases we use
the fact that $\beta = \alpha^{p+1}\in G$.

\medskip

\noindent {\bf Case of}  $K\cong\Z_{p^2}$.\\
In this case we chose $\beta\in\aut{K}$ defined via $\beta(x)=(1+p)x$. The permutation $\gamma(x):=\beta(x) + 1 = (1 + p)x + 1$ belongs to the group $G$ because $\gamma=\wh{1}\beta$.  A direct check shows that $\gamma^p(x) = x + p$ implying $o(\gamma^p)=p$, and, consequently, $o(\gamma) = p^2$. Therefore $\sg{\gamma}$ is a regular cyclic subgroup of $G$ different from $\wh{K}$.

\medskip

\noindent {\bf Case of} $K\cong\Z_p^2$.\\
In this case we chose $\beta\in\aut{\Z_p^2}$ defined via $\beta((x,y))=(x+y,y)$. Then the group $G$ contains the subgroup
$\wh{\Z_p^2}\rtimes\sg{\beta}$ which  consists of all permutations of the form $(x,y)\mapsto(x+ay +u,y+v)$ where $a,u,v\in\Z_p$. A direct check shows that the permutations $\tau_{a,b}:(x,y)\mapsto (x+ay+b,y+a),a,b\in\Z_p$ form a subgroup, say $T$, of $G$ isomorphic to  $\Z_p^2$. It is easy to check that $T$ acts regularly on $\Z_p^2$.
\end{proof}

\subsection{Sylow $2$-subgroups of CIM-groups}

\begin{prop}\label{cyclic2group}
For every $n \ge 4$ the cyclic group $\Z_{2^n}$ is not a connected CIM-group.
\end{prop}
\begin{proof}
The element $a=1+2^{n-1}\in\Z_{2^n}$ has multiplicative order $2$. Therefore the automorphism $\alpha\in\aut{\Z_{2^n}}$ defined via $\alpha(x)=ax$ has order two as well. We construct an antibalanced Cayley map the automorphism group of which contains the subgroup $\wh{\Z}_{2^n} \rtimes \langle \alpha \rangle$. Let $S=\left\{ 1,-1, 3,-3a, a , -a, 3 a, -3  \right\}$ be a set of $8$ different elements, and let $\rho= (1,-1, 3,-3a, a , -a, 3 a, -3 )$ be an $8$-cycle. The permutation $\alpha$ is an automorphism of the map $M:=\cay{\Z_{2^n},S}{\rho}$, because $\alpha(S) =S$ and $\alpha|_S = \rho^4$. Thus the full autmorphism group $G:=\au(M)$ contains the subgroup $A:=\wh{\Z_{2^n}} \rtimes \langle \alpha \rangle$.

 Straightforward calculation shows that $( \wh{1}\alpha)^2(x)=x+a+1 = x+ 2^{n-1}+2$ implying that $(\wh{1}\alpha)^2$ has order  $2^{n-1}$. Hence the order of $ \wh{1}\alpha$ is $2^n$.
Therefore the subgroup $A$ of $G$ contains at least two regular subgroups isomorphic to $\Z_{2^n}$, both of index two.
These subgroups are not conjugate in $A$, since they are normal in $A$. Thus it is enough to prove that $A=G$. The latter is equivalent to showing that
the point stabilizer of $G_0$ has order two.
Assume, towards a contradiction, that $|G_0| > 2$. The group $G_0$ is cyclic and acts on $S$ faithfully and semi-regularly. Therefore there exists an element $\sigma\in G_0$ such that $\sigma^2 = \alpha$. In particular, $\sigma$ has order $4$. Since $\sigma|_S$ commutes with $\rho$, we conclude that $\sigma|_S=\rho^2 = (1,3,a,3a)(-1,-3a,-a,-3)$.

Consider the subset $T:=\{x\in \Z_{2^n}\,|\, |S\cap (S+x)| = 6\}$\footnote{These are the elements at distance two from $0$ in $\cay{\Z_{2^n}}{S}$, each of them is connected to $0$ by $6$ paths of length two.
}.
Since $\sigma$ is an automorphism of $\cay{\Z_{2^n}}{S}$ stabilizing $0$, it satisfies the equation $\sigma(x + S) = \sigma(x)+S$ for every $x\in\Z_{2^n}$. Thus  $T$ is $\sigma$-invariant. A direct calculation yields us $T=\{2,-2,2+2^{n-1},-2+2^{n-1}\}$.

Consider the set $\sigma(S\setminus (S+2)) = \sigma(\{-3,-3+2^{n-1}\})$. Since $\sigma$ is an automorphism of the graph $\cay{\Z_{2^n}}{S}$, we can write $\sigma(S+2) =S+ \sigma(2)$. Therefore
$$
\sigma(S\setminus (S+2)) = \sigma(\{-3,-3+2^{n-1}\})
\implies
S\setminus (S+\sigma(2))) = \{-1,-1+2^{n-1}\}
$$
Since $T$ is $\sigma$-invariant $\sigma(2)\in T$, none of the elements $t\in T$ satisfies $S\setminus (S+t)) = \{-1,-1+2^{n-1}\}$, a contradiction.

\end{proof}

\begin{prop}\label{2sylow} Let $P\leq H$ be a Sylow $2$-subgroup of an CIM-group $H$. Then $P$ is either elementary abelian or cyclic $C_{2^n}, n\leq 3$ or $Q_8$.
\end{prop}
\begin{proof} Assume that $\exp(P) > 2$. Then $P$ contains a cyclic subgroup $C_4=\sg{c}$ of order $4$. We claim that $P$ doesn't contain the Klein subgroup $K_4\cong\Z_2^2$, Indeed, if $K_4=\{1,u,v,w\} < P$ is the Klein subgroup, then the Cayley map $M(K_4,\{u,v\},(u,v))$ is isomorphic, as a map, to the Cayley map $\cay{C_4,\{c,c^{-1}\}}{(c,c^{-1})}$. Hence there should exists an automorphism $\alpha\in \au(H)$ which maps the first map onto the second one. Since both maps are connected, this would imply $\alpha(C_4) = K_4$, a contradiction.

Thus $P$ does not contain $K_4$. By Burnside's Theorem \cite{burnside}, $P$ is either cyclic or generalized quaternion. If $P$ is cyclic, then by Proposition~\ref{cyclic2group} its order is bounded by $8$.

Assume now that $P$ is  a generalized quaternion group distinct from $Q_8$.
Then $P$ contains a characteristic cyclic subgroup $C=\sg{c}$ of index $2$. Then it follows from Lemma~\ref{lemsubgroup} and Proposition~\ref{cyclic2group} that $|C|\leq 8$.
Together with $P\not\cong Q_8$ we obtain that $|C|=8$, and, consequently $|P|=16$.

Let $a\in P$ denote an element of order $4$ outside of $C$. Then $\sg{a,c^2}\cong Q_8$. Let $\alpha$ be an automorphism of $\sg{a,c^2}$ whose action is described by the formulas $\alpha(a)=c^2$ and $\alpha(c^2) = a^{-1}$. Its orbit $\{a,c^2,a^{-1},c^{-2}\}$ is symmetric and generates $\sg{a,c^2}$.
Therefore $M = \cay{\sg{a,c^2},\{a,c^2,a^{-1},c^{-2}\}}{\alpha}$ is a regular balanced Cayley map with $\au(M)=\sg{a,c^2}\rtimes{\alpha}$.
The element $\wh{a} \alpha\in \sg{a,c^2}\rtimes{\alpha}$ has order $8$ and acts regularly on the point set $\sg{a,c^2}$ of the map $M$. Therefore there exists a regular Cayley map $M'$ over the cyclic group of order $8$ isomorphic to $M$. Thus $M\cong M'=\cay{C,S}{\rho}$ for some $S\subseteq C$ and an appropriate rotation $\rho$.

The generalized quaternion group of order 16 contains both $Q_8$ and $\mathbb{Z}_8$, therefore if $H=Q_{16}$ is a CIM-group, there exists $\beta\in \au(H)$ which maps $M$ on $M'$. But in this case $\sg{a,c^2}\cong C$, a contradiction.
\end{proof}

\section{Proof of Theorem~\ref{111015a}}\label{sectionpositive}

We start with the following
\begin{lem}\label{frob} Let $H$ be a group which admits a decomposition $H = C K$ such that $K\cap C=\{1\}$ and $K\triangleleft H$ and $C=\sg{c}$ is cyclic of odd order $m$. Assume that a there exists a faithful $C$-orbit $O=\{k,k^c,...,k^{c^{m-1}}\}$ such that $\sg{OO^{(-1)}}=K$.
Then $H$ is not a connected CIM-group.
\end{lem}
\begin{proof}
It is sufficient to provide an example of a connected non-CI map over $H$.
Take $S:=cO=\{ck_0,ck_1, ..., ck_{m-1}\}$ where $k_i:=k^{c^i}, i=0,...,m-1$. Then $S^{(-1)}\cap S=\emptyset$ because the images of $S$ and $S^{(-1)}$ in $H/K\cong C$ are $c$ and $c^{-1}$, respectively.

Take a Cayley map $M=\cay{H, S\cup S^{(-1)}}{\rho}$ where
$$\rho = (ck_0, (ck_\ell)^{-1},ck_1, (ck_{\ell + 1 })^{-1}, ... )$$
and $\ell = \frac{m+1}{2}$. Notice that the condition $\sg{OO^{(-1)}}=K$ implies that the map is connected.

It follows from the construction that $\rho^2 = \sigma|_{S\cup S^{(-1)}}$, where $\sigma$ is the inner automorphism of $H$ mapping $x$ to $x^c$. Therefore $\sigma\in\aut{M}$ and $G:=\wh{H}\rtimes\sg{\sigma}\leq \aut{M}$.

In order to build a regular subgroup of $\wh{H}\rtimes{\langle \sigma \rangle }$ different from $\wh{H}$ we notice, first, that this group is isomorphic to a direct product $H\times C$ where the isomorphism is defined via $\psi:\wh{h}\sigma^i \mapsto (hc^i,c^{-i})$. Under this isomorphism the point stabilizer $G_1=\sg{\sigma}$ is mapped onto the subgroup $\psi(G_1) = \{(d^{-1},d)\,|\,d\in C\}$.

Let $\pi:H\rightarrow C$ be a projection homomorphism defined  via $\pi(x k):=x$ for $x\in C$ and $k\in K$. Then $F:=\{(h,\pi(h))\,|\, h\in H\}$ is a subgroup of $H\times C$ which intersects $\psi(G_1)$ trivially. Indeed,
$$(h,\pi(h))\in \psi(G_1)\iff \pi(h)=h^{-1}\implies h\in C\implies \pi(h)=h\implies h=h^{-1}.$$
By assumption $C$ has odd order. Therefore $h=1$.

It follows from $F\cap \psi(G_1)=1$ that $\psi^{-1}(F)$ is a regular subgroup of $G$. Thus $G$ contains two regular subgroups isomorphic to $H$, which are $\wh{H}$ and $\psi^{-1}(F)$. Since $\wh{H}\triangleleft G$, it is not conjugate to $\psi^{-1}(F)$ inside $G$.

Since $G_1$ has two orbits on the connection set $S\cup S^{-1}$, either  $\aut{M}=G$ or $[\aut{M}:G]=2$. In the first case we already have two $H$-regular subgroups of $G$ which are non-conjugate in $G$. In the second case it follows from $\rho(x^{-1})=\rho(x)^{-1}$ that $M$ is a regular balanced map over $H$. It was proved in \cite{skovierasiran} that $\wh{H}\trianglelefteq \aut{M}$. Since $G$ contains a $H$-regular subgroup distinct from $\wh{H}$, it is not conjugate to $\wh{H}$ inside $\aut{M}$.
\end{proof}

{\bf Remark.} The condition $\sg{OO^{(-1)}}=K$ is always fulfilled if $K$ does not contain a proper non-trivial $C$-normalized subgroups. For example, if $K$ is of prime order, then $\sg{OO^{(-1)}}=K$ holds for any non-trivial orbit $O$.

\medskip

Now we are ready to prove Theorem \ref{111015a}.
\begin{proof}
Let $T$ denote a Sylow $2$-subgroup of $H$. Our proof is divided into few steps.

\medskip

\noindent{\bf Step 1.} Any normal subgroup $N$ of $H$ of odd order  is cyclic.\\
Since all Sylow subgroups of $N$ have prime order, it is sufficient to prove that any Sylow subgroup of $N$ is normal in $H$. This would follow if we prove that each Sylow subgroup of $N$ has a normal complement. To show that let us fix a Sylow subgroup $P$ of order $p$, where $p$ is prime. By Burnside Theorem the existence of a normal complement follows from $\mathbf{N}_N(P) = \mathbf{C}_N(P)$. Assume towards a contradiction that there exists $g\in\mathbf{N}_N(P)$ which does not centralize $P$. We may assume that $o(g)$ is a prime power. By Lemma~\ref{sylowodd} any Sylow subgroup of $N$ has a prime order. Therefore $o(g)$ is prime distinct from $p$.  In this case the group $\sg{g}P$ satisfies the assumptions of Lemma~\ref{frob} and therefore, is not a connected CIM group. A contradiction.\

\medskip

\noindent {\bf Step 2.} $T$ has a normal complement. \\
By Proposition~\ref{2sylow},
$T$ is isomorphic to one of the groups $\Z_2^r,\Z_4,\Z_8$ or $Q_8$.
If $T$ is cyclic, then the result follows from the Cayley normal $2$-complement Theorem.

Assume now that $T$ is not cyclic, i.e. $T\cong\Z_2^r$ or $T\cong Q_8$.
By Frobenius normal $p$-complement Theorem it is sufficient to show that $\mathbf{N}_H(T)/\mathbf{C}_H(T)$ is a $2$-group. Notice that $\mathbf{N}_H(T)/\mathbf{C}_H(T)$ is embedded into $\aut{T}$.

If $T \not\cong Q_8$, then $T\cong\Z_2^e$ for some $e \geq 1$. Assume, towards a contradiction, that $\mathbf{N}_H(T)/\mathbf{C}_H(T)$ is not a $2$-group. Then there exists an element $g\in \mathbf{N}_H(T)$ of odd order which acts on $T$ nontrivially. Without loss of generality, we may assume that $o(g)$ is a $p$-power for some odd prime divisor $p$ of $|H|$.
Since $|H|_p=p$, we conclude $o(g)=p$. Since $T$ is elementary abelian $2$-group, it contains a minimal $g$-invariant subgroup $T_1$ on which $g$ acts non-trivially. The group $\sg{g} T_1$ satisfies the assumptions of Proposition~\ref{frob}. Therefore $\sg{g} T_1$ is not a connected CIM-group. A contradiction.

If $T\cong Q_8$ and $\mathbf{N}_H(T)/\mathbf{C}_H(T)$ is not a $2$-group, then this group contains an element of order $3$. Hence $\mathbf{N}_H(T)$ contains an element $g$ of order $3$ which acts on $T$ non-trivially. Applying Lemma~\ref{frob} once more we get a contradiction.

\medskip

\noindent{\bf Step 3.}  If $T$ is non-cyclic, then $H\cong N\times T$.\\
As it was mentioned before, a CIM-group has the property that any two elements of the same order are either conjugate or inverse conjugate by an automorphism of $H$. In particular, this implies that all involtiuons of $H$ are $\aut{H}$-conjugate.

If $T$ is non-cyclic, then either it is elementary abelian or $Q_8$. Let us assume first that $T$ is an elementary abelian $2$-group of order at least $4$. Then all non-trivial elements of $T$ are $\aut{H}$-conjugate. Since $N$ is characteristic in $H$, the subgroups $\mathbf{C}_N(s)$ and $\mathbf{C}_N(t)$ are $\aut{H}$-conjugate for  any $s\neq t\in T\setminus\{1\}$. Since any subgroup of $N$ is characteristic in $H$, we conclude that $K:=\mathbf{C}_N(s) = \mathbf{C}_N(t) = \mathbf{C}_N(ts)$. Let $L\leq N$ be a unique subgroup complementary to
$K$ in $N$. Then both $t$ and $s$ invert the elements of $L$, Therefore $st$ acts trivially on $L$ implying $L\leq K$, and consequently $L=1$. Thus
any element of $T$ centralizes $N$. Therefore $H\cong N\times T$.

It remains to settle the case when $T\cong Q_8$. In this case all cyclic subgroups of order $4$ are $\aut{H}$-conjugate. Since $\aut{N}$ is abelian the commutator subgroup $Z$ of $T$ acts trivially on $N$. The quotient group $\bar{H}=H/Z$ is isomorphic to $N\rtimes\Z_2^2$. Moreover all involutions of $\Z_2^2$ are $\aut{\bar{H}}$-conjugate. From the previous paragraph we obtain that $\Z_2^2$ acts trivially on $N$. So, the semi-direct product $N\rtimes\Z_2^2$ is, in fact, the direct one. Therefore $H\cong N\times T$.
\end{proof}

\section{Proof of Theorem~\ref{111015c}}

We start with introduccing the notation $\mathscr{M}$ for the set of groups $\Z_n\times\Z_2^r, \Z_n\times\Z_4,\Z_n\times Q_8$ where $n$ is a square-free odd number.
The statement below collects the properties of these groups. We omit the proof because it is straightforward.
\begin{prop}\label{071215a}
The following properties hold
\begin{enumerate}
\item The subgroups and factor groups of $H\in\mathscr{M}$ belong to
$\mathscr{M}$;
\item Any two subgroups $A,B\leq H\in\mathscr{M}$ of the same order are conjugate by an automorphism of $H$;
\item Any subgroup automorphism $\beta\in\aut{A}, A\leq H$ may be extended to  an automorphsim of $H$;
\item The groups in $\mathscr{M}$ are Hamiltonian.
\end{enumerate}
\end{prop}

Our first step provides a reduction of Theorem~\ref{111015c} to the connected case.
\begin{prop}\label{071215b}
If the groups of $\mathscr{M}$ are connected CIM-group, then they are CIM groups.
\end{prop}
\begin{proof}
Let $M=\cay{H,S}{\rho}$ and $M'=\cay{H,S'}{\rho'}$ be two isomorphic map over a group $H\in\mathscr{M}$. Then $|\sg{S}|=|\sg{S'}|$, and by Proposition~\ref{071215a} there exists an automorphism $\alpha\in\aut{H}$ such that $\alpha(\sg{S'}) = \sg{S}$. Thus replacing $M'$ by $\alpha(M')$ we may assume that $\sg{S} = \sg{S'}$.
Since $M$ and $M'$ are isomorphic, their connected components $M_1 := \cay{\sg{S},S}{\rho}$ and $M_1':=\cay{\sg{S},S'}{\rho'}$ are isomorphic too. Both $M_1$ and $M_1'$ are connected maps over the group $\sg{S}\in\mathscr{M}$. Therefore there exists $\beta\in\aut{\sg{S}}$ such that $\beta(M_1)=M_1'$.
By Proposition~\ref{071215a} $\beta$ can be extended up  to an automorphism of $H$, $\alpha$ say. Then $\alpha(M)=M'$ hereby proving the claim.
\end{proof}

To prove Theorem~\ref{111015c} for connected
maps we provide a little bit more general result.

\begin{Thm}\label{100915c} Let $G\leq\sym{\Omega}$ be a transitive permutation group with cyclic point stabilizer which contains a regular  subgroup $H\in\mathscr{M}$. Then any $H$-regular subgroup of $G$ is conjugate to $H$ in $G$.
\end{Thm}

If $H$ is abelian then  by Ito's theorem \cite{ito} the group $G$ is metabelian and therefore $G$ is solvable. If $H$ is non-abelian, then it is nilpotent and $G$ is solvable by Kegel-Wielandt theorem.

We will prove Theorem~\ref{100915c} by induction on $|G|$ and assume that $G$ is a counterexample of a minimal order. In particular, this implies that the theorem is correct for any proper subgroup $X$ where $H\leq X < G$. Since $G$ is a counerexample, there exists an $H$-regular subgroup $F$ of $G$  which is not conjugate to $\wh{H}$ inside $G$. We fix $F$ till the end of the proof. By the minimality of $G$, we may assume that $\sg{\wh{H},F^g}=G$ for each $g\in G$. 
We write the order of $H$ by $2^rn $. Recall that $n$ is an odd square-free number.

Below the following notiation is used. If $G$ is a group acting on a set $X$, then $G_X$ denote the kernel of this action and $G^X$ denote the image of $G$ in $\sym{X}$.

\begin{prop}\label{170915a} Let $G$ be a minimal counterexample to Theorem~\ref{100915c} and $\mathcal{D}$ be a proper non-trivial imprimitivity system of $G$. Then $G^{\mathcal{D}} = H^{\mathcal{D}}$, or equivalently, $G=H G_{\mathcal{D}}\iff G_\omega\leq G_{{\mathcal D}}$.
\end{prop}
\begin{proof} Notice that $\mathcal{D}$ is an imprimitivity system of $H$ too. Since $H$ is regular, the setwise stabilizer $H_{\{D\}}$ of a block  $D\in\mathcal{D}$ acts regularly on $D$. Since the block stabilizers are conjugate in $H$ and $H$ is a Hamiltonian group, the subgroup $H_{\{D\}}$ does not depend on a choice of $D\in\mathcal{D}$. Therefore
the subgroup $H_{\{D\}}, D\in\mathcal{D}$ coincides with $H_{\mathcal{D}}$ implying that $D$ is an orbit of $H_{\mathcal{D}}$. It follows from $H_{\mathcal{D}}\leq G_{\mathcal{D}}$ that $ G_{\mathcal{D}}$ acts transitively on each block of $\mathcal{D}$. The group $H^{\mathcal{D}}$ is a regular subgroup of $G^{\mathcal{D}}$. Also $H^{\mathcal{D}}\cong H/H_{\mathcal{D}}\in\mathscr{M}$.
The point stabilizer of $G^{\mathcal{D}}$ is isomorphic to $G_\omega G_{\mathcal{D}}/ G_{\mathcal{D}}\cong G_\omega /(G_\omega\cap G_{\mathcal{D}})$, and, therefore, is cyclic. Thus $G^{\mathcal{D}}\leq\sym{\mathcal{D}}$ satisfies the assumptions of Theorem~\ref{100915c}. Since $|G^{\mathcal{D}}|=|G|/|G_{\mathcal{D}}| < |G|$, we may apply the induction hypothesis to $G^{\mathcal{D}}$. It yields us that  $F^{\mathcal{D}}$ and $H^{\mathcal{D}}$ are conjugate in $G^{\mathcal{D}}$. Therefore there exists $g\in G$ such that $(F^g)^{\mathcal{D}}=H^{\mathcal{D}}$ implying $G^{\mathcal{D}} = \sg{F^g,H}^{\mathcal{D}} =
 \sg{(F^g)^{\mathcal{D}},H^{\mathcal{D}}} = H^{\mathcal{D}}$.
\end{proof}
\begin{prop}\label{170915b} Let $G$ be a minimal counterexample to Theorem~\ref{100915c}. Then $G$ admits at most one minimal imprimitivity system.
\end{prop}
\begin{proof}
Assume, towards a contradiction, that $G$ admits two minimal imprimitivity systems, say $\mathcal{D}$ and $\mathcal{E}$. By Proposition~\ref{170915a}   $G_\omega\leq G_{\mathcal{D}}$ and $G_\omega\leq G_{\mathcal{E}}$. If follows from minimality of $\mathcal{E}$ and $\mathcal{D}$ that $G_{\mathcal{D}}\cap G_{\mathcal{E}}=\{1\}$. Therefore $G_\omega=\{1\}$ implying $G=H$ contrary to $G$ being a counterexample.
\end{proof}
For a set of elements $S$ of a group acting on a set $X$, we denote by $\fix(S)$, the elements of $X$ fixed by every $s \in S$.
\begin{prop}\label{170915c} Let $G\leq\sym{\Omega}$ be a transitive permutation group with cyclic point stabilizer. Then for each $S\leq G_\omega$ the set $\fix(S)$ is a block of $G$.
\end{prop}
\begin{proof} Assume that $\fix(S)\cap\fix(S^g)$ is non-empty, and pick an arbitrary $\delta\in\fix(S)\cap\fix(S^g)$. Then $S,S^g\leq G_\delta$. Since $G_\delta$ is cyclic, any two subgroups of $G_\delta$ of the same order coincide. Therefore $S=S^g$ implying that $\fix(S^g)=\fix(S)$.
\end{proof}

\noindent{\bf Proof of Theorem~\ref{100915c}}\\
Let $\mathcal{P}$ be a minimal imprimitivity system of $G$.
Pick an arbitrary block $\Pi\in\mathcal{P}$. Then $G_{\{\Pi\}}^\Pi$ is a solvable primitive permutation subgroup of $\sym{\Pi}$. Therefore $|\Pi|$
is a power of a prime divisor $p$ of $|H|$.

\medskip

We split the proof into few steps.

\medskip

\noindent {\bf Step 1.} $|\mathcal{P}| > 1$.\\
Assume the contrary, that is $|\mathcal{P}|=1$, or, equivalently, $\Pi=\Omega$. In this case $H$ is a $p$-group. By Proposition~\ref{170915c} the set $\fix(G_{\alpha,\beta})$ is a block of  $G$ for any pair of points $\alpha,\beta\in\Omega$. Together with primitivity of $G$ this implies that $G_{\alpha,\beta}=1$ whenever $\alpha\neq\beta$. Therefore $G$  is a Frobenius group the kernel of which, $K$ say, has order $|H|$. Since $K$ is a unique Sylow $p$-subgroup of $G$, we conclude $H=K=F$, a contradiction.

\medskip

\noindent {\bf Step 2.} $G_{\mathcal{P}}$ is a $p$-group.\\
Assume that there exists a prime divisor $q\neq p$ of $G_{\mathcal{P}}$. Since $G_{\mathcal{P}}$ acts transitively on each block $\Pi\in\mathcal{P}$ and $G_\omega\leq  G_{\mathcal{P}}$, we conclude that $|G_{\mathcal{P}}| = |G_\omega|\cdot|\Pi|$. This implies that $q$ divides $|G_\omega|$. Thus $G_\omega$ contains a subgroup $Q$ of order $q$.  By Proposition~\ref{170915c} the set $\fix(Q)$ is block of $G$. It follows from Proposition~\ref{170915a} that $Q\leq G_\omega\leq G_{\mathcal{P}}$ that $Q$ fixes each block of $\mathcal{P}$ setwise. Since blocks of $\mathcal{P}$ have a $p$-power size,
the set $\fix(Q)$ intersects each block of $\mathcal{P}$ non-trivially. By Proposition~\ref{170915b} $\mathcal{P}$ is a unique minimal imprimitivity system of $G$. Therefore each block of $G$ is a union of some blocks of $\mathcal{P}$. Thus $\fix(Q)=\Omega$ implying that $Q=\{1\}$. A contradiction.

\medskip

\noindent {\bf Step 3.} $G_{p'}=H_{p'} = \mathbf{O}_{p'}(G)\neq\{1\}$.\\
 By Step 2 $G_{\mathcal{P}}$ is a $p$-group. Therefore
$|H|_{p'}=|H^\mathcal{P}|_{p'}$ and $|G|_{p'}=|G^\mathcal{P}|_{p'}$.
By Proposition~\ref{170915a} $H^\mathcal{P} = G^\mathcal{P}\cong G/G_\mathcal{P}$. Therefore $|G_{p'}|=|H_{p'}|$ implying $G_{p'}=H_{p'}$. By Hall's Theorem there exists $g\in G$ such that $(F^g)_{p'}=(F_{p'})^g = H_{p'}$ implying that $F^g$ normalizes $H_{p'}$. Combining this with $G = \sg{H,F^g}$ we conclude that $H_{p'}\trianglelefteq G$. Together with $H_{p'}=G_{p'}$ we obtain that
$H_{p'}=G_{p'} = \mathbf{O}_{p'}(G)$.

If $G_{p'}$ is trivial, then $G$ and $H$ are $p$-groups.  Since $\mathcal{P}$ is non-trivial (by minimality) and $|\mathcal{P}| > 1$, we conclude that $|H| = |\Omega|\geq p^2$. Together with $H\in\mathscr{M}$ this implies that $p=2$ and $H$ is one of the groups: $\Z_2^r,\Z_4,Q_8$. Since $G^\Pi_{\{\Pi\}}$ is a primitive $2$-group, we conclude that $|\Pi|=2$.  Therefore $G_{\mathcal{P}}$ is an elementary abelian $2$-group. By Proposition~\ref{170915a} $G_\omega\leq G_{\mathcal{P}}$. Therefore $|G_\omega|=2$ and both $H$ and $F$ are index two subgroups of $G$. So both of them are normal in $G$ and $G\leq \mathbf{N}_{\sym{\Omega}}(H)$. If $H$ is isomorphic to one of $\Z_2,\Z_4, Q_8$, then $H$ is a unique $H$-regular subgroup of $\mathbf{N}_{\sym{\Omega}}(H)$, contrary to $F\leq G\leq  \mathbf{N}_{\sym{\Omega}}(H)$.
Therefore $H\cong\Z_2^r,r\geq 2$.

It follows from $H\neq F$ that $G=HF$ and $H\cap F\leq\mathbf{Z}(G)$. It follows from $G=FH$ that a unique involution $s\in G_\omega$ has a presentation $s=h_0f_0$ with $h_0\in H$ and $f_0\in F$. Notice that $h_0\not\in H\cap F$ and $f_0\not\in F\cap H$ (otherwise we would have $s\in (H\cup F)\setminus\{1\}$ which cannot happen because $(H\cup F)\setminus\{1\}$ contains only fixed-point-free permutations).
Thus $G = HF=\sg{f_0}\sg{h_0}(H\cap F)$. It follows from $s^2 = 1$ that $[f_0,h_0]=1$. Together with $H\cap F\leq\mathbf{Z}(G)$ we conclude that $G$ is an abelian group. Thus $G$ should be regular contrary to $|G_\omega|=2$.

\medskip

\noindent {\bf Step 4. Getting the final contradiction.} It follows from Step 3 that
$\mathbf{O}_{p'}(G)$ is nontrivial. Therefore the orbits of $\mathbf{O}_{p'}(G)$ form a non-trivial imprimitivity system of $G$ with block size coprime to $p$.
Since $\mathcal{P}$ is a unique
minimal imprimitivity system (Proposition~\ref{170915b}), the orbits of $\mathbf{O}_{p'}(G)$  are unions of blocks of $\mathcal{P}$.  But the this is impossible, since the cardinality of blocks of $\mathcal{P}$ is a $p$-power.
\qed

\bigskip

We finish this section by resolving the status of the cyclic group of order $8$.
\begin{prop}\label{111015b} A cyclic group $\Z_8$ is a CIM-group.
\end{prop}
\begin{proof}
Assume towards a contradiction that $M:=\cay{\Z_8,S}{\rho}$ is a non-CI map over $\Z_8$. Let $P$ be a Sylow $2$-subgroup of  $G:=\aut{M}$ which contains $\wh{\Z_8}$. Then $P$ contains a regular cyclic subgroup which is not conjugate to $\wh{\Z_8}$ inside $P$. In particular, $|P| \geq 16$. Therefore $|\mathbf{N}_P(\wh{\Z_8})|\geq 16$. The point stabilizer   $\mathbf{N}_P(\wh{\Z_8})_0$ is cyclic and is contained in $\aut{\Z_8}$. Therefore
$|\mathbf{N}_P(\wh{\Z_8})_0|=2$, or, equivalently, $\mathbf{N}_P(\wh{\Z_8})_0 =\sg{\alpha}$ for some $\alpha\in\aut{\Z_8}$.

If $\wh{\Z_8}$ is a unique regular cyclic subgroup of $\mathbf{N}_P(\wh{\Z_8})$, then $\mathbf{N}_P(\mathbf{N}_P(\wh{\Z_8}))$ normalizes $\wh{Z_8}$. So, in this case $\mathbf{N}_P(\mathbf{N}_P(\wh{\Z_8})) = \mathbf{N}_P(\wh{\Z_8})$ implying $P = \mathbf{N}_P(\wh{\Z_8})$, because in a $p$-group the normalizer of a proper subgroup is strictly bigger than the subgroup. The latter equality contradicts our assumption that $P$ contains non-conjugate regular cyclic subgroups. Thus $\mathbf{N}_P(\wh{\Z_8}) = \wh{\Z_8}\rtimes\sg{\alpha}$ contains non-conjugate regular cyclic subgroups. This yields a unique choice for $\alpha\in\aut{H}$, namely: $\alpha(x)=5x,x\in\Z_8$. Notice that $\wh{\Z_8}\rtimes\sg{\alpha}$ contains exactly two regular cyclic subgroups
$\wh{\Z_8}$ and $\sg{\wh{1}\alpha}$. Each of these subgroups is normal in $\wh{\Z_8}\rtimes\sg{\alpha}$.

Since $\alpha\in G_0$, it acts semiregularly on $S$. Combining this with $\sg{S}=\Z_8$ and $S = -S$ we obtain that
the only possibility for $S$ is $\{1,5,3,7\}$. It follows from $\rho^2 = \alpha|_S$ that either $\rho = (1,3,5,7)$ or $\rho = (1,7,5,3)$. In both cases $M$ is an antibalanced map the full automorphism group of which has order $32$ and has a decomposition $G = \wh{\Z_8}\sg{\rho}$ where $\rho$ acts trivially on the subgroup $2\Z_8$. In both cases all regular cyclic subgroups are conjugate in $G$.
\end{proof}

\section{Acknowledgments}

The second author was supported by a postdoctoral fellowship
funded by the Skirball Foundation via the Center of Advanced Studies
at the Ben-Gurion University of the Negev, and, in addition,
by the research grant with PI Dmitry Kerner. The first author was supported by the Israeli  Ministry of Absorption.

Both authors are grateful personally to Mikhail Klin who created the opportunity to work together. Without his efforts this cooperation would not led to writing this paper.

\end{document}